\documentclass[12pt,DIV=14]{scrartcl}
\usepackage{amsmath,amsthm,amssymb}
\usepackage{enumerate}
\usepackage%
[colorlinks=true,citecolor=red,urlcolor=purple,linkcolor=blue,pdfborder={0 0 0}]%
{hyperref}
\usepackage[pdftex]{graphicx}
\DeclareGraphicsExtensions{.jpg,.pdf,.mps,.png}
\graphicspath{{images/}}
\allowdisplaybreaks
\newtheorem{theorem}{Theorem}[section]
\newtheorem{corollary}[theorem]{Corollary}
\newtheorem{lemma}[theorem]{Lemma}
\theoremstyle{definition}

\theoremstyle{remark}
\newtheorem{remark}{Remark}

\providecommand{\xx}{\mathbf x }
\providecommand{\uu}{\mathbf u }
\providecommand{\vv}{\mathbf v }
\providecommand{\ww}{\mathbf w }

\providecommand{\N}{\mathbb{N}}
\providecommand{\R}{\mathbb{R}}

\providecommand{\eps}{\varepsilon}

\providecommand{\BB}{{\mathbf B }}
\providecommand{\CC}{{\mathbf C }}
\providecommand{\DD}{{\mathbf D }}
\providecommand{\EE}{{\mathbf E }}

\providecommand{\HH}{{\mathbf H }}
\providecommand{\II}{{\mathbf I }}
\providecommand{\JJ}{{\mathbf J }}

\providecommand{\LL}{{\mathbf L }}

\providecommand{\PP}{{\mathbf P }}
\providecommand{\RR}{{\mathbf R }}

\providecommand{\UU}{{\mathbf U }}
\providecommand{\VV}{{\mathbf V }}
\providecommand{\WW}{{\mathbf W }}

\providecommand{\epsbm}{{\mathbf \varepsilon }}
\providecommand{\PPhi}{{\mathbf \Phi }}
\providecommand{\PPsi}{{\mathbf \Psi }}

\providecommand{\PPi}{{\mathbf \Pi }}
\providecommand{\nubm}{\mathbf \nu }
\providecommand{\omegabm}{{\mathbf \omega }}

\newcommand{\curl}{\textnormal{curl}}
\newcommand{\divv}{\textnormal{div}}

\providecommand{\LMEspace}{\LL_2(\Omega)}
\providecommand{\LMHspace}{\HH(\curl,\Omega)}

\providecommand{\LMEspacew}{\LL_{2,\eps(\EE)}(\Omega)}
\providecommand{\LMHspacew}{\LL_2(\Omega)}

\providecommand{\LMEspacet}{\LL_2(\Omega)}
\providecommand{\LMHspacet}{\HH(\curl,\Omega)}

\providecommand{\NEspace}{\HH_0(\curl,\Omega)}

\providecommand{\NBspace}{\HH(\divv,\Omega)}

\providecommand{\NEspacew}{\LL_{2,\eps(\EE)}(\Omega)}
\providecommand{\NHspacew}{\LL_2(\Omega)}
\providecommand{\NBspacew}{\LL_{\mu^{-1}}^{2}(\Omega)}

\providecommand{\NEspacet}{\HH_0(\curl,\Omega)}
\providecommand{\NHspacet}{\LL_2(\Omega)}

\providecommand{\Th}{{\mathcal{T}_h }}

\title{Semi-discrete finite element approximation applied to Maxwell's equations
in nonlinear media}
\author{Lutz Angermann\footnote{%
Institut f\"ur Mathematik, Technische Universit\"at Clausthal,
Erzstra{\ss}e 1, D-38678 Clausthal-Zellerfeld, Germany,
e-mail: lutz.angermann@tu-clausthal.de}}
\date{January 9, 2019}

\begin{document}
\maketitle
\begin{abstract}
In this paper the semi-discrete finite element approximation of initial boundary value problems for
Maxwell's equations in nonlinear media of Kerr-type is investigated.
For the case of N\'ed\'elec elements from the first family,
a priori error estimates are established for the approximation.
\end{abstract}



\section{Introduction}
In this paper we investigate the semi-discrete conforming finite element approximation
to the solution of Maxwell's equations for nonlinear media of Kerr-type.
As a concrete example, we consider the (meanwhile classical) N\'ed\'elec elements
from the so-called first family.
To the best knowlege of the author, the nonlinear situation is not yet
well investigated, most works dealing with nonlinear effects are computational
or experimental (see, e.g., \cite{Angermann:18d}).
Here we derive energy (stability) estimates for the weakly formulated
problem and error estimates for the semi-discretized problem.

Let $Q_T:=(0,T)\times\Omega$, where $\Omega\subset\R^3$ is a simply connected domain
with a sufficiently smooth boundary $\partial\Omega$ and $T>0$ is the length of the time
interval under consideration.
Let $\DD,\BB,\EE,\HH:\;Q_T\to\R^3$
represent the displacement field, the magnetic induction,
the electric and magnetic field intensities, respectively.
The time-dependent Maxwell's equations in a nonlinear medium can be written in the form
\begin{align}
\partial_t\DD-\nabla\times\HH&=0 \,\textnormal{ in } Q_T,\label{1.1}\\
\partial_t\BB+\nabla\times\EE &= 0 \,\textnormal{ in } Q_T,\label{1.2}
\end{align}
where the following constitutive relations hold:
\begin{align}\label{de:D}
\BB&:=\mu_0\HH,
\quad
\DD:=\eps_0\EE+\PP(\EE).
\end{align} 
Here $\eps_0>0$ and $\mu_0>0$ are the vacuum permittivity and the permeability, respectively.
Often the constitutive relation for the polarization $\PP=\PP(\EE)$ is approximated
by a truncated Taylor series \cite{Boyd:03}.
In the case of an isotropic material, it takes the form
\begin{align*}
\PP(\EE)&:=\eps_0\Big( \chi^{(1)}\EE
+\chi^{(3)}|\EE|^{2}\EE\Big),
\end{align*}
where $\chi^{(j)}:\;\Omega\to\R$ are the media susceptibility coefficients,
$j=1,3$.
Then from (\ref{de:D}) we obtain the representation
\begin{equation*}
\DD=\eps_0\eps_s\EE
\quad\text{with}\quad
\eps_s=\eps_s(\EE):=1+\chi^{(1)}+\chi^{(3)}|\EE|^2,
\end{equation*}
and it follows by a simple calculation that
\[
\partial_t\DD = 2\eps_0\chi^{(3)}(\EE\cdot\partial_t\EE)\EE
+ \eps_0\eps_s \partial_t\EE
= \eps_0\left(\eps_s\II + \epsbm_m\right)\partial_t\EE
\quad\text{with}\quad
\epsbm_m=\epsbm_m(\EE):=2\chi^{(3)}\EE\EE^\top,
\]
where $\II$ denotes the identy in $\R^d$.
Setting
\[
\eps(\EE):=\eps_0\left(\eps_s(\EE)\II  + \epsbm_m(\EE)\right),
\]
the system (\ref{1.1})--(\ref{de:D}) can be written as
\begin{align}
\eps(\EE)\partial_t\EE-\nabla\times \HH=0 \,\textnormal{ in } Q_T,\label{non:1}\\
\mu_0\partial_t\HH+\nabla\times\EE = 0 \,\textnormal{ in } Q_T.\label{non:2}
\end{align}
Next we state a simple result
which in particular implies that the matrix $\eps(\EE)$ is regular
for all electric field intensities $\EE$ under consideration.
\begin{lemma}\label{l:epsEEposdef}
Let $\chi^{(1)}, \chi^{(3)}\ge 0$ a.e.\ in $\Omega$.
Then the symmetric matrix $\eps(\PPsi)$ is uniformly positive definite
a.e.\ for any $\PPsi\in\R^3$.
\end{lemma}
\begin{proof}
For all $\PPhi\in\R^3$, it holds that
\[
\eps_0^{-1}\PPhi^\top\eps(\PPsi)\PPhi
=\Big(1+\chi^{(1)}+\chi^{(3)}|\PPsi|^2\Big) |\PPhi|^2 + 2\chi^{(3)}|\PPsi\cdot\PPhi|^2
\ge|\PPhi|^2.
\]
\end{proof}
As in \cite{Pototschnig:09}, we denote the inverse by
\[
\CC(\EE):=\big(\eps(\EE)\big)^{-1}.
\]
By means of the Sherman-Morrison formula (see, e.g., \cite{Golub:96b}),
the matrix $\CC(\EE)$ can be given explicitely:
\begin{equation}\label{eq:ShermanMorrison}
\CC(\EE)=\frac{1}{\eps_0}\CC_m(\EE)
\end{equation}
with
\begin{align*}
\CC_m(\EE)&:=\left(\eps_s(\EE)\II  + \epsbm_m(\EE)\right)^{-1}
=\frac{1}{\eps_s(\EE)}\left(\II-\frac{1}{\eps_s(\EE)+2\chi^{(3)}|\EE|^2}\,\epsbm_m(\EE)\right)\\
&=\frac{1}{\eps_s(\EE)}\left(\II-\frac{1}{1+\chi^{(1)}+3\chi^{(3)}|\EE|^2}\,\epsbm_m(\EE)\right).
\end{align*}
Therefore, if the formula (\ref{eq:ShermanMorrison}) holds,
the system (\ref{non:1})--(\ref{non:2}) takes the form
\begin{align}
\eps_0\partial_t\EE-\CC_m(\EE)\nabla\times \HH=0 \,\textnormal{ in } Q_T,\label{non:3}\\
\mu_0\partial_t\HH+\nabla\times\EE = 0 \,\textnormal{ in } Q_T.\label{non:4}
\end{align}
This is an appropriate formulation for the development of time-discrete numerical algorithms,
see, e.g., \cite{Pototschnig:09}.

A perfect conducting boundary condition on $\Omega$ is assumed so that
\begin{align}
\nubm\times\EE=0 \quad\textnormal{on }(0,T)\times\partial\Omega,\label{1.5}
\end{align}
where $\nubm$ denotes the outward unit normal on $\partial\Omega$.
In addition, initial conditions have to be specified so that
\begin{align}
\EE(0,\xx)=\EE_0(\xx)\quad\textnormal{and }\quad\HH(0,\xx)=\HH_0(\xx) 
\quad\textnormal{for all }
\xx\in\Omega, \label{1.6}
\end{align}
where $\EE_0:\;\Omega\to\R^3$ and $\HH_0:\;\Omega\to\R^3$
are given functions, and $\HH_0$ satisfies
\begin{align}
\nabla\cdot(\mu_0\HH_0)=0 \quad\textnormal{in }\Omega,
\quad
\HH_0\cdot\nubm=0 \quad\textnormal{on }\partial\Omega.\label{1.7}
\end{align}
The divergence-free condition in (\ref{1.7}) together with (\ref{non:2}) implies that
\begin{align}
\nabla\cdot(\mu_0\HH)&=0 \quad\textnormal{in } Q_T,\label{1.8}
\end{align}

\section{Notation}
For a real number $p\in[1,\infty]$, the symbol $L^p(\Omega)$ denotes the usual
Lebesgue spaces equipped with the norm $\|\cdot\|_{L^p(\Omega)}$.
The analogous spaces of vector fields $\uu:\Omega\to\R^3$ are denoted by
$\LL_p(\Omega):=\left[L^p(\Omega)\right]^3$ with the norm $\|\cdot\|_{\LL^p(\Omega)}$.

In what follows we have to deal with weighted function spaces.
Given a weight
$\omegabm:\;\Omega\to\R$,
where the values of $\omegabm$ are positive
a.e.\ on $\Omega,$ we define a weighted inner product and a weighted norm by
\begin{equation}\label{eq:L2wip}
(\uu,\vv)_{\omegabm}:=\int_{\Omega}\omegabm\,\uu\cdot\vv\, d\xx
\quad\text{and}\quad
\|\uu\|_{\omegabm}:=\|\uu\|_{\LL_{\omegabm}^2(\Omega)}:=\sqrt{(\uu,\uu)_{\omegabm}}.
\end{equation}
The space $\LL_{\omegabm}^2(\Omega)$ consists of vector fields
$\uu:\Omega\to\R^3$ with Lebesgue-measurable components and such that
\[
\|\uu\|_{\omegabm}<\infty.
\]
In the case $\omegabm= 1,$ the subscript is omitted.
An elementary property of weighted spaces, which we will apply at
different places without special emphasis, is the monotonicity
w.r.t.\ the weight:
If $\omegabm_1$, $\omegabm_2$ are two weights such that $\omegabm_1\le\omegabm_2$
a.e.\ on $\Omega,$ then
\[
\|\uu\|_{\LL_{\omegabm_1}^2(\Omega)}\le\|\uu\|_{\LL_{\omegabm_2}^2(\Omega)}
\quad\text{for all }
\uu\in\LL_{\omegabm_1}^2(\Omega)\cap\LL_{\omegabm_2}^2(\Omega).
\]
As transient problems are addressed, we will work with functions
that depend on time and have values in certain Banach spaces.
If $\uu=\uu(t,\xx)$ is a vector field of the space variable $\xx$
and the time variable $t$, it is suitable to separate these variables in such a way that
$\uu(t)=\uu(t,\cdot)$ is considered as a function of $t$ with values
in a Banach space, say $X$, with the norm $\|\cdot\|_{X}$.
That is, for any $t\in(0,T)$, the mapping $\xx\mapsto\uu(t,\xx)$ is 
interpreted as a parameter-dependent element $\uu(t)$ of $X$.
In this sense we will write $\EE(t)=\EE(t,\cdot)$, $\HH(t)=\HH(t,\cdot)$
and so on.

The space $\CC^{m}(0,T;X)$, $m\in \N\cup\{0\}$, consists of all
continuous functions $\uu:\;(0,T)\to X$ that have continuous derivatives up 
to order $m$ on $(0,T).$ It is equipped with the norm
\begin{align*}
\|\uu\|_{\CC^{m}(0,T;X)}:=\sum_{j=0}^{m}\sup_{t\in(0,T)}\|\uu^{(j)}(t)\|_{X}.
\end{align*} 
For the sake of consistency in the notation we will write
$\CC(0,T;X):=\CC^{0}(0,T;X)$.

The space $\LL_p(0,T;X)$ with $p\in[1,\infty)$ contains (equivalent classes of)
strongly measurable functions $\uu:\;(0,T)\to X$ such that
\begin{align*}
\int_0^T\|\uu(t)\|_{X}^p dt<\infty
\end{align*}
(for the definition of strongly measurable functions we refer to \cite{Kufner:77}).
The norm on $\LL_p(0,T;X)$ is defined by
\begin{align*}
\|\uu\|_{\LL_p(0,T;X)}:=\left\{\int_0^T\|\uu(t)\|_{X}^p dt\right\}^{1/p}.
\end{align*}
These spaces can be equipped with a weight, too. In particular, 
we will write
\begin{align*}
\|\uu\|_{\LL_2(0,T;\LL_\omega^2(\Omega))}
:=\left\{\int_0^T\int_{\Omega}|\uu(t)|^2\,\omega\, d\xx dt\right\}^{1/2}.
\end{align*}
Finally, all the above definitions can be extended to the standard
Sobolov spaces of functions with weak 
spatial derivatives of maximal order $r\in\N$ in $L^p(\Omega)$:
$W^{r,p}(\Omega)$ with norm $\|\cdot\|_{W^{r,p}(\Omega)}$.
If $p=2$, we write 
$H^{r}(\Omega):=W^{r,2}(\Omega)$ and 
$\|\cdot\|_{H^{r}(\Omega)}:=\|\cdot\|_{W^{r,2}(\Omega)}$. 

The space $H_0^1(\Omega)$ is defined as the closure of 
$C_0^{\infty}(\Omega)$ with respect to the norm
$\|\cdot\|_{H^1(\Omega)}$,
where $C_0^{\infty}(\Omega)$ denotes the space of
all arbitrarily often differentiable functions
with compact support on $\Omega$.
It is well knwon that $H_0^1(\Omega)$ is a closed subspace of $H^1(\Omega)$
and consists of elements $u$ such that $u=0$ on $\partial\Omega$ in the sense of 
traces \cite{Adams:03}.
As in the case of the $L^p$-spaces, we shall write
$\WW^{r,p}(\Omega):=[W^{r,p}(\Omega)]^3$ and so on.

Furthermore, we need the following Hilbert spaces that are related to 
the (weak) rotation and divergence operators:
\begin{align*}
\HH(\curl,\Omega)&:=\{\uu\in \LL_2(\Omega):\ \nabla\times\uu\in\LL_2(\Omega)\},\\
\HH_0(\curl,\Omega)&:=\{\uu \in\HH(\curl,\Omega):\ \uu\times\nubm|_{\partial\Omega}=0\},\\
\HH(\divv,\Omega)&:=\{\uu\in \LL_2(\Omega):\ \nabla\cdot\uu\in\LL_2(\Omega)\}.
\end{align*}
These spaces are equipped with the norms (resp.\ induced norms)
\begin{align*}
\|\uu\|_{\HH(\curl,\Omega)}&:=\big\{\|\uu\|_0^2 
+\|\nabla\times\uu\|_0^2\big\}^{1/2},\\
\|\uu\|_{\HH(\divv,\Omega)}&:=\big\{\|\uu\|_0^2+\|\nabla\cdot\uu\|_0^2\big\}^{1/2}.
\end{align*}
We refer to \cite{Duvaut:76}, \cite{Raviart:77}, \cite{Girault:86} and
\cite{Ciarlet:02b}
for details about these spaces.

\section{Weak formulations}
We assume that a solution
\[
(\EE,\HH)\in\big(\CC^1(0,T;\LMEspacew)\cap \CC(0,T;\NEspace)\big)
\times\big(\CC^1(0,T;\LMHspacew)\cap \CC(0,T;\LMHspace)\big)
\]
of the nonlinear Maxwell's equations (\ref{non:1})--(\ref{non:2}) exists and is unique.

We multiply equation (\ref{non:1}) by a test function
$\PPsi\in\LMEspacet$ and integrate over $\Omega$.
Similarly we multiply (\ref{non:2}) by a test function
$\PPhi\in\LMHspacet$, integrate the result over $\Omega$
and integrate by parts the second term.
This shows that it is natural to look for a weak solution
$(\EE,\HH)\in\big(\CC^1(0,T;\LMEspacew)\cap \CC(0,T;\LMEspace)\big)
\times\big(\CC^1(0,T;\LMHspacew)\cap \CC(0,T;\LMHspace)\big)$
of (\ref{non:1})--(\ref{non:2}) such that
\begin{align}
(\eps(\EE)\partial_t\EE,\PPsi)-(\nabla\times\HH,\PPsi)
&=0
&\quad\forall\,\PPsi\in\LMEspacet,\label{3.1a}\\ 
(\mu_0\partial_t\HH,\PPhi)+(\EE,\nabla\times\PPhi)
&=0
&\quad\forall\,\PPhi\in\LMHspacet.\label{3.1b}
\end{align}
Alternatively, the use of test functions $\PPsi\in\NEspacet$
and $\PPhi\in\NHspacet$
and the integration by parts in the equation (\ref{non:1})
leads to the notion of a weak solution
$(\EE,\HH)\in\big(\CC^1(0,T;\NEspacew)\cap \CC(0,T;\NEspace)\big)
\times\big(\CC^1(0,T;\NHspacew)\big)$
of (\ref{non:1})--(\ref{non:2}) such that
\begin{align}
(\eps(\EE)\partial_t\EE,\PPsi)-(\HH,\nabla\times\PPsi)
&=0
&\forall\,\PPsi\in\NEspacet,\label{eq:ned_weak}\\ 
(\mu_0\partial_t\HH,\PPhi)+(\nabla\times\EE,\PPhi)
&=0
&\forall\,\PPhi\in\NHspacet.\label{eq:ravi_weak}
\end{align}
In both cases, the initial conditions (\ref{1.6}) have to be satisfied
at least in the sense of
\linebreak[4]
$\CC(0,T;\LL_2(\Omega))$.
\begin{remark}
As a consequence of the embedding (as sets)
\[
[C_0^{\infty}(\Omega)]^3 \subset [C^{\infty}(\Omega)\cap H^1(\Omega)]^3
\subset [H^1(\Omega)]^3 \subset \HH(\divv, \Omega) \subset \LL_2(\Omega)
\]
and of the fact that $C_0^{\infty}(\Omega)$ is dense in $L^{2}(\Omega)$
we see that $\HH(\divv, \Omega)$ is a dense subset of $\LL_2(\Omega)$
\cite{Adams:03}.
Therefore the test space in (\ref{eq:ravi_weak}) can be reduced to $\NBspace$.
Also note that $\HH\in \CC(0,T;\NBspace)$ due to (\ref{1.8}).
\end{remark}
\begin{remark}
In the case where $\mu_0$ is not a constant but a highly variable function $\mu=\mu(\xx)$
it is more convenient to use the magnetic flux density $\BB=\mu\HH$ instead of $\HH$
as a dependent variable \cite{Makridakis:95}.
In such a case, the formulation (\ref{eq:ned_weak})--(\ref{eq:ravi_weak})
is replaced by
\begin{align}
(\eps(\EE)\partial_t\EE,\PPsi)-(\mu^{-1}\BB,\nabla\times\PPsi)
&=0
&\forall\,\PPsi\in\NEspacet,\label{eq:makrned_weak}\\ 
(\mu^{-1}\partial_t\BB,\PPhi)+(\nabla\times\EE,\mu^{-1}\PPhi)
&=0
&\forall\,\PPhi\in\NBspace,\label{eq:makrravi_weak}
\end{align}
where
\[
(\EE,\BB)\in\big(\CC^1(0,T;\NEspacew)\cap \CC(0,T;\NEspace)\big)
\times\big(\CC^1(0,T;\NBspacew)\big)\cap \CC(0,T;\NBspace)\big).
\]
\end{remark}
Next we will formulate a stability result for the problem (\ref{3.1a})--(\ref{3.1b}).
For this, we extend this problem
to the case of a nontrivial right-hand side for a moment.
\begin{theorem}\label{th:Nstab}
Let $\JJ_e\in\LL_2(0,T;\LL_{2,\eps_0^{-1}(1+\chi^{(1)})^{-1}}(\Omega))$,
$\JJ_m\in\LL_2(0,T;\LL_{2,\mu_0^{-1}}(\Omega))$
-- the electric and magnetic current densities, respectively -- be given
and assume that the system
\begin{align}
(\eps(\EE)\partial_t\EE,\PPsi)-(\nabla\times\HH,\PPsi)
&=-(\JJ_e,\PPsi)
&\quad\forall\,\PPsi\in\LMEspacet,\label{3.1n}\\ 
(\mu_0\partial_t\HH,\PPhi)+(\EE,\nabla\times\PPhi)
&=-(\JJ_m,\PPhi)
&\quad\forall\,\PPhi\in\LMHspacet\label{3.1m}
\end{align}
together with the initial conditions (\ref{1.6})
has a weak solution
\[
(\EE,\HH)\in\big(\CC^1(0,T;\NEspacew)\cap \CC(0,T;\NEspace)\big)
\times\big(\CC^1(0,T;\NHspacew)\big).
\]
If
\[
W(t):=\frac{1}{2}\Big[
\|\EE(t)\|_{\eps_0(1+\chi^{(1)})}^2
+ \frac{3}{2}\big\||\EE(t)|^2\big\|_{\eps_0\chi^{(3)}}^2
+\|\HH(t)\|_{\mu_0}^2\Big]
\]
denotes the nonlinear electromagnetic energy at the time $t$,
the following energy law in differential form
holds:
\begin{equation}\label{eq:nemel1}
\frac{dW}{dt}=-(\JJ_e,\EE)-(\JJ_m,\HH),
\quad
t\in(0,T).
\end{equation}
\end{theorem}
\begin{proof}
Taking $\PPsi=\EE$ and $\PPhi=\HH$ in (\ref{3.1n})--(\ref{3.1m})
and adding the result gives
\begin{equation}\label{eq:multadd1}
(\eps(\EE)\partial_t\EE,\EE)+(\mu_0\partial_t\HH,\HH)
=-(\JJ_e,\EE)-(\JJ_m,\HH)
\end{equation}
An elementary calculation shows that
\begin{align*}
\eps(\EE)\partial_t\EE
&=\eps_0\left((1+\chi^{(1)}+\chi^{(3)}|\EE|^2)\II + 2\chi^{(3)}\EE\EE^\top\right)\partial_t\EE\\
&=\eps_0(1+\chi^{(1)}+\chi^{(3)}|\EE|^2)\partial_t\EE + 2\eps_0\chi^{(3)}\EE(\EE\cdot\partial_t\EE)\\
&=\eps_0(1+\chi^{(1)}+\chi^{(3)}|\EE|^2)\partial_t\EE + \eps_0\chi^{(3)}\EE\partial_t|\EE|^2,
\end{align*}
hence
\begin{align*}
\EE\cdot(\eps(\EE)\partial_t\EE)
&=\eps_0(1+\chi^{(1)}+\chi^{(3)}|\EE|^2)\EE\cdot\partial_t\EE
+ \eps_0\chi^{(3)}\EE\cdot\EE\partial_t|\EE|^2\\
&=\frac{\eps_0}{2}(1+\chi^{(1)}+\chi^{(3)}|\EE|^2)\partial_t|\EE|^2
+ \eps_0\chi^{(3)}|\EE|^2\partial_t|\EE|^2\\
&=\frac{\eps_0}{2}(1+\chi^{(1)})\partial_t|\EE|^2
+ \frac{3\eps_0}{2}\chi^{(3)}|\EE|^2\partial_t|\EE|^2\\
&=\frac{\eps_0}{2}(1+\chi^{(1)})\partial_t|\EE|^2
+ \frac{3\eps_0}{4}\chi^{(3)}\partial_t|\EE|^4.
\end{align*}
Then it follows from equation (\ref{eq:multadd1}) that
\[
\frac{dW}{dt}
=\frac{1}{2}\frac{d}{dt}\Big[
\|\EE\|_{\eps_0(1+\chi^{(1)})}^2
+ \frac{3}{2}\big\||\EE|^2\big\|_{\eps_0\chi^{(3)}}^2
+\|\HH\|_{\mu_0}^2\Big]
=-(\JJ_e,\EE)-(\JJ_m,\HH).
\]
\end{proof}
\begin{corollary}\label{cor:Nstab}
Under the assumptions of Thm.\ \ref{th:Nstab}, for all $t\in(0,T)$,
the following estimate is valid:
\[
W(t)\le 2 W(0) 
+t\Big[\|\JJ_e\|_{\LL_2(0,T;\LL_{2,\eps_0^{-1}(1+\chi^{(1)})^{-1}}(\Omega))}^{2}
+\|\JJ_m\|_{\LL_2(0,T;\LL_{2,\mu_0^{-1}}(\Omega))}^{2}\Big].
\]
\end{corollary}
\begin{proof}
We estimate the right-hand side of (\ref{eq:nemel1})
by means of the Cauchy-Bunyakovski-Schwarz' inequality
(twice -- in the integral version and in the finite sum version):
\begin{align*}
|(\JJ_e,\EE) &+ (\JJ_m,\HH)|
\le\|\JJ_e\|_{\eps_0^{-1}(1+\chi^{(1)})^{-1}}\|\EE\|_{\eps_0(1+\chi^{(1)})}
+\|\JJ_m\|_{\mu_0^{-1}}\|\HH\|_{\mu_0}\\
&\le\left\{\|\JJ_e\|_{\eps_0^{-1}(1+\chi^{(1)})^{-1}}^{2}
+\|\JJ_m\|_{\mu_0^{-1}}^{2}\right\}^{1/2}
\left\{\|\EE\|_{\eps_0(1+\chi^{(1)})}^{2}+\|\HH\|_{2,\mu_0}\right\}^{1/2}\\
&\le \sqrt{2}\left\{\|\JJ_e\|_{\eps_0^{-1}(1+\chi^{(1)})^{-1}}^{2}
+\|\JJ_m\|_{\mu_0^{-1}}^{2}\right\}^{1/2}
\sqrt{W(t)}.
\end{align*}
Since $\frac{d}{dt}\sqrt{W}=\frac{1}{2\sqrt{W}}\frac{dW}{dt}$ for $W(t)>0$
(the case $W(t)=0$ is trivial),
it follows from (\ref{eq:nemel1}) that
\[
\sqrt{2}\,\frac{d}{dt}\sqrt{W(t)}
\le\left\{\|\JJ_e\|_{\eps_0^{-1}(1+\chi^{(1)})^{-1}}^{2}
+\|\JJ_m\|_{\mu_0^{-1}}^{2}\right\}^{1/2}.
\]
Integrating this inequality w.r.t.\ $t$, we get
\[
\sqrt{2}\sqrt{W(t)} \le \sqrt{2}\sqrt{W(0)}
+\int_0^t\left\{\|\JJ_e\|_{\eps_0^{-1}(1+\chi^{(1)})^{-1}}^{2}
+\|\JJ_m\|_{\mu_0^{-1}}^{2}\right\}^{1/2}ds.
\]
Then
\begin{align*}
W(t) &\le 2 W(0) 
+\left(\int_0^t\left\{\|\JJ_e\|_{\eps_0^{-1}(1+\chi^{(1)})^{-1}}^{2}
+\|\JJ_m\|_{\mu_0^{-1}}^{2}\right\}^{1/2}ds\right)^{2}\\
&\le2 W(0) 
+t\int_0^t\left[\|\JJ_e\|_{\eps_0^{-1}(1+\chi^{(1)})^{-1}}^{2}
+\|\JJ_m\|_{\mu_0^{-1}}^{2}\right]ds\\
&\le 2 W(0) 
+t\Big[\|\JJ_e\|_{\LL_2(0,T;\LL_{2,\eps_0^{-1}(1+\chi^{(1)})^{-1}}(\Omega))}^{2}
+\|\JJ_m\|_{\LL_2(0,T;\LL_{2,\mu_0^{-1}}(\Omega))}^{2}\Big].
\end{align*}
\end{proof}
\begin{remark}
Analogous results as in Thm.\ \ref{th:Nstab} and Cor.\ \ref{cor:Nstab} can be obtained for
the corresponding ``non-homogeneous'' version of (\ref{eq:ned_weak})--(\ref{eq:ravi_weak})
and for the subsequent semi-discretizations.
\end{remark}

\section{Spatial discretization}
\subsection{Semi-discretization of the weak formulations}
Let $\WW_h\subset\LMEspace$,
$\UU_h\subset\LMHspace$,
$\UU_{0h}\subset\NEspace$,
and $\VV_h\subset\NBspace$ be finite-dimensional subspaces.

The semi-discrete (in space) problem for the system (\ref{3.1a})--(\ref{3.1b}) consists
in determining elements $(\EE_h,\HH_h)\in\CC^1(0,T;\WW_h)\times \CC^1(0,T;\UU_h)$
such that
\begin{align}
(\eps(\EE_h)\partial_t\EE_h, \PPsi_h)-(\nabla\times\HH_h,\PPsi_h)
&=0
&\forall\,\PPsi_h\in\WW_h\label{3.6a},\\
(\mu_0\partial_t\HH_h,\PPhi_h)+(\EE_h,\nabla\times\PPhi_h)
&=0
&\forall\,\PPhi_h\in\UU_h\label{3.6b}.
\end{align}
For the the equations (\ref{eq:ned_weak})--(\ref{eq:ravi_weak}),
the semi-discrete problem involves the determination of elements
$(\EE_h,\HH_h)\in \CC^1(0,T;\UU_{0h})\times \CC^1(0,T;\VV_h)$
such that
\begin{align}
(\eps(\EE_h)\partial_t\EE_h, \PPsi_h)-(\HH_h,\nabla\times\PPsi_h)
&=0
&\forall\,\PPsi_h\in\UU_{0h}\label{eq:dis_ned},\\
(\mu_0\partial_t\HH_h,\PPhi_h)+(\nabla\times\EE_h,\PPhi_h)
&=0
&\forall\,\PPhi_h\in\VV_h\label{eq:dis_ravi}.
\end{align}
The initial conditions for both problems read formally as
\begin{align*}
\EE_h(0,\xx)=\EE_{0h}(\xx)\quad\textnormal{and }\quad\HH_h(0,\xx)=\HH_{0h}(\xx) 
\quad\textnormal{for all }
\xx\in\Omega,
\end{align*}
where the concrete requirements to the particular choice of the discrete initial data
$(\EE_{0h},\HH_{0h})$ will be seen later (Thm.\ \ref{th:sdeLM}).

\subsection{The choice of the finite element spaces}
In the rest of the paper we will choose the so-called first family
of N\'ed\'elec edge elements, usually denoted by $\DD_{k}$ and $\RR_{k}$ for $k\in\N$,
for the construction of the concrete finite element spaces 
(for details see \cite{Nedelec:80} or \cite[Ch.~5]{Monk:03}).
That is, given an arbitrary member $\Th$ of a family of triangulations of $\Omega$
consisting of open tetrahedra $K$, we set
\begin{align*}
\UU_h&:=\{\ww\in\HH(\curl;\Omega):\;\ww|_K\in\RR_k\ \forall K\in\Th\},\\
\VV_h&:=\{\ww\in\HH(\divv;\Omega):\;\ww|_K\in\DD_k\ \forall K\in\Th\},\\
\WW_h&:=\{\ww\in\LL_2(\Omega):\;\ww|_{K}\in\PP_{k-1}\ \forall K\in\Th\},
\end{align*}
where $\PP_\ell:=[\mathcal{P}_\ell]^3$ and $\mathcal{P}_\ell$ is the space
of scalar real-valued polynomials
in three variables of maximal degree $\ell\in\N\cup\{0\}$.
To deal with the case of the N\'ed\'elec formulation (\ref{eq:ned_weak})--(\ref{eq:ravi_weak}),
we still have to introduce the space
$\UU_{0h}:=\UU_h\cap\NEspace$.
In the subsequent error analysis, we will make use of some projection operators.
For the Lee-Madsen formulation (\ref{3.6a})--(\ref{3.6b}),
we need projections
$\PP_h:\;\LMEspace\to\WW_h$, $\Pi_h:\;\LMHspace\to\UU_h$.

Let $\PP_h$ be the standard $\LL_2(\Omega)$-projection operator onto $\WW_h$,
i.e.\ for given $\ww\in\LL_2(\Omega)$ the image $\PP_h\ww\in\WW_h$ is defined by
\begin{equation}\label{eq:L2projection}
(\PP_h\ww, \PPsi_h)=(\ww, \PPsi_h), \quad\forall\,\PPsi\in\WW_h.
\end{equation}
For this operator the following standard error estimate holds:
If $\ww\in\HH^k(\Omega)$, then
\begin{align}\label{eq:errL2projection}
\|\ww-\PP_h\ww\|\le Ch^k\|\ww\|_{\HH^k(\Omega)}.
\end{align}
Moreover, since
$\nabla\times\UU_h\subset\WW_h$ (see the beginnings of the proofs of \cite[Thm.\ 3.3]{Monk:91}
or \cite[Lemma 5.40]{Monk:03}),
it holds that
\begin{equation}\label{eq:L2projection1}
(\PP_h\ww, \nabla\times\PPhi_h)=(\ww, \nabla\times\PPhi_h), \quad\forall\,\PPhi\in\UU_h.
\end{equation}
Next, for $\vv\in\LMHspace$ we define $\PPi_h\vv\in\UU_h$ by
\begin{align}
\big(\nabla\times\PPi_h\vv,\nabla\times\PPsi_h\big)
&=\big(\nabla\times\vv,\nabla\times\PPsi_h\big)
\quad
\forall\,\PPsi_h\in\UU_h,\label{eq:pi1a}\\
(\PPi_h\vv,\nabla p_h)&=(\vv,\nabla p_h)
\quad
\forall p_h\in S_h^k,\label{eq:pi1b}
\end{align}
where $S_h^k$ is defined as
\[
S_h^k:=\{v\in H_1(\Omega)/\R:\; v|_{K}\in \mathcal{P}_k
\ \forall K\in \mathcal{T}_h\},
\]
see \cite[Subsect.\ 4.2]{Monk:91}.

If $\vv\in\HH^{k+1}(\Omega)$ such that $\nabla\cdot\vv=0$ in $\Omega$
and $\nubm\cdot\vv=0$ on $\partial\Omega$, then there exists a constant $C>0$ such that
\begin{equation}\label{eq:errorpi1}
\|\vv-\PPi_h\vv\|
\le Ch^k\|\vv\|_{\HH^{k+1}(\Omega)}
\end{equation}
(see \cite[Thm.\ 4.6]{Monk:91}).

\section{An error estimate for the semi-discrete problem}

In this section we formulate and prove the main result.
\begin{theorem}[Semi-discrete error estimate for the Lee-Madsen formulation]
\label{th:sdeLM}
Let $k\in\N$,
$\chi^{(1)},\chi^{(3)}\in L_\infty(\Omega)$,
$\EE_0\in\LL_\infty(\Omega),$
$\HH_0\in\LL_{2,\mu_0}(\Omega)$ satisfying (\ref{1.7}),
\[
\big(\EE, \HH\big)\in
\big(\CC(0,T;\LL_\infty(\Omega)\cap\NEspace)\cap\EE\in\CC^1(0,T;\HH^k(\Omega))\big)
\times \CC^1(0,T;\HH^{k+1}(\Omega))
\]
be the weak solution of the system (\ref{3.1a})--(\ref{3.1b}), and
\[
\big(\EE_h,\HH_h\big)\in\CC(0,T;\WW_h)\cap\CC^1(0,T;\LL_\infty(\Omega))
\times\CC(0,T;\UU_h)
\]
be the finite element solution of the system (\ref{3.6a})--(\ref{3.6b}) respectively,
where the inclusion is to be understood uniformly w.r.t.\ the mesh parameter $h$
in the sense that
$\|\EE_h\|_{\CC^1(0,T;\LL_\infty(\Omega))}$ is bounded by a constant independent of $h$.
Then there exists a constant $C>0$ independent of $h$
such that the following error estimate holds:
\begin{align*}
&\|\EE_h(T)-\EE(T)\|_{\eps_{0}}+\|\HH_h(T)-\HH(T)\|_{\mu_{0}} \nonumber\\
&\le C \left[\|\PP_h \EE_{0}-\EE_{0h}\|_{\eps_{0}}
+\|\Pi_h \HH_{0}-\HH_{0h}\|_{\mu_{0}}+h^k\right]
\end{align*}
(the detailed structure of the bound is given at the end of the proof).
\end{theorem}
\begin{proof}
We set $\PPsi:=\PPsi_h\in\WW_h$ in (\ref{3.1a}) and
$\PPhi:=\PPhi_h\in\UU_h$ in (\ref{3.1b}):
\begin{align*}
(\eps(\EE)\partial_t\EE,\PPsi_h)-(\nabla\times\HH,\PPsi_h)
&=0
&\quad\forall\,\PPsi_h\in\WW_h, \\
(\mu_0\partial_t\HH,\PPhi_h)+(\EE,\nabla\times\PPhi_h)
&=0
&\quad\forall\,\PPhi_h\in\UU_h.
\end{align*}
By means of the projection operators $\PP_h$ and $\PPi_h$
defined in (\ref{eq:L2projection}) and (\ref{eq:pi1a})--(\ref{eq:pi1b}), resp.,
from this we get
\begin{align}
(\eps(\EE)\partial_t\EE,\PPsi_h)-(\nabla\times\PPi_h\HH,\PPsi_h)
&=(\nabla\times(\HH-\PPi_h\HH),\PPsi_h)
&\forall\,\PPsi_h\in\WW_h, \label{3.1dta}\\
(\mu_0\partial_t\PPi_h\HH,\PPhi_h)+(\PP_h\EE,\nabla\times\PPhi_h)
&=\mu_0(\PPi_h\partial_t\HH-\partial_t\HH,\PPhi_h) \nonumber\\
&\quad +\mu_0(\partial_t\PPi_h\HH-\PPi_h\partial_t\HH,\PPhi_h) \label{3.1dtb}\\
&\quad +(\PP_h\EE-\EE,\nabla\times\PPhi_h)
&\forall\,\PPhi_h\in\UU_h.
\nonumber
\end{align}
The last term on the right-hand side of (\ref{3.1dta}) vanishes thanks to
the properties of $\PPi_h$, see \cite[eq.\ (2.4)]{Monk:91}
and (\ref{eq:pi1a}).

The second term on the right-hand side of (\ref{3.1dtb}) can be omitted because of the
commutation property $\partial_t\PPi_h\HH=\PPi_h\partial_t\HH$,
which results from the continuity properties of the operator $\PPi_h$.
The last term on the right-hand side vanishes thanks to
the property (\ref{eq:L2projection1}) of $\PP_h$.

Therefore (\ref{3.1dta})--(\ref{3.1dtb}) simplify to
\begin{align}
(\eps(\EE)\partial_t\EE,\PPsi_h)-(\nabla\times\PPi_h\HH,\PPsi_h)
&=0
&\quad\forall\,\PPsi_h\in\WW_h, \label{3.1dtc}\\
(\mu_0\partial_t\PPi_h\HH,\PPhi_h)+(\PP_h\EE,\nabla\times\PPhi_h)
&=\mu_0(\PPi_h\partial_t\HH-\partial_t\HH,\PPhi_h)
&\quad\forall\,\PPhi_h\in\UU_h.  \label{3.1dtd}
\end{align}
Now, subtracting (\ref{3.1dtc})--(\ref{3.1dtd}) from the system (\ref{3.6a})--(\ref{3.6b})
and taking into consideration that $\mu_0$ is constant,
we obtain:
\begin{align}
(\eps(\EE_h)\partial_t\EE_h-\eps(\EE)\partial_t\EE,\PPsi_h)
-(\nabla\times(\HH_h-\PPi_h\HH),\PPsi_h)
&=0
&\forall\,\PPsi_h\in\WW_h, \label{3.6ea}\\
\mu_0(\partial_t(\HH_h-\PPi_h\HH),\PPhi_h)
+(\EE_h-\PP_h\EE,\nabla\times\PPhi_h) \nonumber\\
=\mu_0(\partial_t\HH-\PPi_h\partial_t\HH,\PPhi_h)
&&\forall\,\PPhi_h\in\UU_h. \label{3.6eb}
\end{align}
Now we will deal with the first term of (\ref{3.6ea}),
where we have in mind the choice $\PPsi_h=\EE_h-\PP_h\EE$ in what follows:
\begin{align*}
&\eps_0^{-1}\left[\eps(\EE_h)\partial_t\EE_h-\eps(\EE)\partial_t\EE\right]\\
&=\left(\eps_s(\EE_h)\II + \epsbm_m(\EE_h)\right)\partial_t\EE_h
-\left(\eps_s(\EE)\II + \epsbm_m(\EE)\right)\partial_t\EE)\\
&=\left((1+\chi^{(1)}+\chi^{(3)}|\EE_h|^2)\II + \epsbm_m(\EE_h)\right)\partial_t\EE_h
-\left((1+\chi^{(1)}+\chi^{(3)}|\EE|^2)\II + \epsbm_m(\EE)\right)\partial_t\EE\\
&=(1+\chi^{(1)})\partial_t\EE_h-(1+\chi^{(1)})\partial_t\EE\\
&\quad +\left(\chi^{(3)}|\EE_h|^2\II + \epsbm_m(\EE_h)\right)\partial_t\EE_h
-\left(\chi^{(3)}|\EE|^2\II + \epsbm_m(\EE)\right)\partial_t\EE\\
&=(1+\chi^{(1)})\partial_t(\EE_h-\EE)
+\chi^{(3)}\left[|\EE_h|^2\partial_t\EE_h-|\EE|^2\partial_t\EE\right]\\
&\quad +2\chi^{(3)}\left[\EE_h\EE_h^\top\partial_t\EE_h-\EE\EE^\top\partial_t\EE\right]
=:\delta_1+\delta_2+\delta_3.
\end{align*}
The treatment of $\delta_1$ is quite obvious. With $\EE_h-\EE=\PPsi_h+\PP_h\EE-\EE$
we get
\begin{align*}
\delta_1=(1+\chi^{(1)})\partial_t\PPsi_h+(1+\chi^{(1)})\partial_t(\PP_h\EE-\EE)
=:\delta_{11}+\delta_{12}.
\end{align*}
The term $\delta_2$ is decomposed as follows:
\begin{align*}
\delta_2
&=\chi^{(3)}\left[|\EE_h|^2\partial_t\EE_h-|\EE|^2\partial_t\EE\right]\\
&=\chi^{(3)}\left[|\EE_h|^2-|\EE|^2\right]\partial_t\EE_h
+\chi^{(3)}|\EE|^2\partial_t(\EE_h-\EE)\\
&=\chi^{(3)}(\EE_h+\EE)^\top(\EE_h-\EE)\partial_t\EE_h
+\chi^{(3)}|\EE|^2\partial_t\PPsi_h
+\chi^{(3)}|\EE|^2\partial_t(\PP_h\EE-\EE)\\
&=:\delta_{21}+\delta_{22}+\delta_{23}.
\end{align*}
For $\delta_3$, we use the following decomposition:
\begin{align*}
\delta_3
&=2\chi^{(3)}\left[\EE_h\EE_h^\top\partial_t\EE_h-\EE\EE^\top\partial_t\EE\right]\\
&=2\chi^{(3)}\left[\EE_h\EE_h^\top-\EE\EE^\top\right]\partial_t\EE_h
+2\chi^{(3)}\EE\EE^\top\partial_t(\EE_h-\EE)\\
&=2\chi^{(3)}(\EE_h-\EE)\EE_h^\top\partial_t\EE_h
+2\chi^{(3)}\EE(\EE_h-\EE)^\top\partial_t\EE_h\\
&\quad +2\chi^{(3)}\EE\EE^\top\partial_t\PPsi_h
+2\chi^{(3)}\EE\EE^\top\partial_t(\PP_h\EE-\EE)\\
&=:\delta_{31}+\delta_{32}+\delta_{33}+\delta_{34}.
\end{align*}
With these decompositions, equation (\ref{3.6ea}) takes the form
\begin{align*}
&(\eps(\EE_h)\partial_t\EE_h-\eps(\EE)\partial_t\EE,\PPsi_h)
-(\nabla\times(\HH_h-\PPi_h\HH),\PPsi_h)\\
&=\eps_0\int_\Omega\left[
\delta_{11}+\delta_{22}+\delta_{33}\right]^\top\PPsi_hd\xx
+\eps_0\int_\Omega\left[\delta_{12}+\delta_{21}+\delta_{23}
+\delta_{31}+\delta_{32}+\delta_{34}\right]^\top\PPsi_hd\xx\\
&\quad -(\nabla\times(\HH_h-\PPi_h\HH),\PPsi_h)=0,
\end{align*}
or, after some rearrangement,
\begin{align}
&\eps_0\int_\Omega\left[\delta_{11}+\delta_{22}+\delta_{33}\right]^\top\PPsi_hd\xx
-(\nabla\times(\HH_h-\PPi_h\HH),\PPsi_h)\nonumber\\
&=-\eps_0\int_\Omega\left[\delta_{12}+\delta_{21}+\delta_{23}
+\delta_{31}+\delta_{32}+\delta_{34}\right]^\top\PPsi_hd\xx.
\label{eq:3.6eb}
\end{align}
Then:
\begin{align*}
&\eps_0\int_\Omega\left[\delta_{11}+\delta_{22}+\delta_{33}\right]^\top\PPsi_hd\xx\\
&=\eps_0\int_\Omega\left[(1+\chi^{(1)})\partial_t\PPsi_h^\top\PPsi_h
+\chi^{(3)}|\EE|^2\partial_t\PPsi_h^\top\PPsi_h
+2\chi^{(3)}\left(\EE\EE^\top\partial_t\PPsi_h\right)^\top\PPsi_h
\right]d\xx\\
&=\frac{\eps_0}{2}\int_\Omega\left[(1+\chi^{(1)})\partial_t|\PPsi_h|^2
+\chi^{(3)}|\EE|^2\partial_t|\PPsi_h|^2
+4\chi^{(3)}\EE^\top\partial_t\PPsi_h\EE^\top\PPsi_h
\right]d\xx\,.
\end{align*}
Since
\[
|\EE|^2\partial_t|\PPsi_h|^2
=\partial_t(|\EE|^2|\PPsi_h|^2)-\partial_t(|\EE|^2)|\PPsi_h|^2
\quad\text{and}\quad
\EE^\top\partial_t\PPsi_h
=\partial_t(\EE^\top\PPsi_h)-\partial_t\EE^\top\PPsi_h,
\]
it follows that
\begin{align*}
&\eps_0\int_\Omega\left[\delta_{11}+\delta_{22}+\delta_{33}\right]^\top\PPsi_hd\xx\\
&=\frac{\eps_0}{2}\int_\Omega\left[(1+\chi^{(1)})\partial_t|\PPsi_h|^2
+\chi^{(3)}\partial_t(|\EE|^2|\PPsi_h|^2)
-\chi^{(3)}\partial_t(|\EE|^2)|\PPsi_h|^2\right.\\
&\quad +\left. 4\chi^{(3)}\partial_t(\EE^\top\PPsi_h)\EE^\top\PPsi_h
-4\chi^{(3)}\partial_t\EE^\top\PPsi_h\EE^\top\PPsi_h\right]d\xx\\
&=\frac{\eps_0}{2}\int_\Omega\left[(1+\chi^{(1)})\partial_t|\PPsi_h|^2
+\chi^{(3)}\partial_t(|\EE|^2|\PPsi_h|^2)
-\chi^{(3)}\partial_t(|\EE|^2)|\PPsi_h|^2\right.\\
&\quad +\left. 2\chi^{(3)}\partial_t|\EE^\top\PPsi_h|^2
-4\chi^{(3)}\partial_t\EE^\top\PPsi_h\EE^\top\PPsi_h\right]d\xx\\
&=\frac{\eps_0}{2}\int_\Omega\left[(1+\chi^{(1)})\partial_t|\PPsi_h|^2
+\chi^{(3)}\partial_t(|\EE|^2|\PPsi_h|^2)
+2\chi^{(3)}\partial_t|\EE^\top\PPsi_h|^2\right]d\xx\\
&\quad -\frac{\eps_0}{2}\int_\Omega
\chi^{(3)}\partial_t(|\EE|^2)|\PPsi_h|^2d\xx
-2\eps_0\int_\Omega
\chi^{(3)}\partial_t\EE^\top\PPsi_h\EE^\top\PPsi_hd\xx\,.
\end{align*}
From the estimates
\begin{align*}
&\left|\frac{\eps_0}{2}\int_\Omega
\chi^{(3)}\partial_t(|\EE|^2)|\PPsi_h|^2d\xx\right|
=\left|\eps_0\int_\Omega
\chi^{(3)}\partial_t\EE^\top\EE|\PPsi_h|^2d\xx\right|\\
&\le \|\chi^{(3)}\|_{L_\infty(\Omega)}\|\partial_t\EE\|_{\CC(0,T;\LL_\infty(\Omega))}
\|\EE\|_{\CC(0,T;\LL_\infty(\Omega))}\|\PPsi_h\|_{\eps_0}^2\\
&\le \|\chi^{(3)}\|_{L_\infty(\Omega)}
\|\EE\|_{\CC^1(0,T;\LL_\infty(\Omega))}^2\|\PPsi_h\|_{\eps_0}^2
\end{align*}
and, analogously,
\[
\left|\eps_0\int_\Omega
\chi^{(3)}\partial_t\EE^\top\PPsi_h\EE^\top\PPsi_hd\xx\right|
\le \|\chi^{(3)}\|_{L_\infty(\Omega)}
\|\EE\|_{\CC^1(0,T;\LL_\infty(\Omega))}^2\|\PPsi_h\|_{\eps_0}^2
\]
we conclude that
\begin{align}
&\eps_0\int_\Omega\left[\delta_{11}+\delta_{22}+\delta_{33}\right]^\top\PPsi_hd\xx\nonumber\\
&\ge\frac{\eps_0}{2}\int_\Omega\left[(1+\chi^{(1)})\partial_t|\PPsi_h|^2
+\chi^{(3)}\partial_t(|\EE|^2|\PPsi_h|^2)
+2\chi^{(3)}\partial_t|\EE^\top\PPsi_h|^2\right]d\xx\nonumber\\
&\quad -3\|\chi^{(3)}\|_{L_\infty(\Omega)}
\|\EE\|_{\CC^1(0,T;\LL_\infty(\Omega))}^2\|\PPsi_h\|_{\eps_0}^2\nonumber\\
&=\frac{1}{2}\partial_t\|\PPsi_h\|_{\eps_0(1+\chi^{(1)})}^2
+\frac{\eps_0}{2}\partial_t\int_\Omega\chi^{(3)}\left[|\EE|^2|\PPsi_h|^2
+2|\EE^\top\PPsi_h|^2\right]d\xx\nonumber\\
&\quad -3\|\chi^{(3)}\|_{L_\infty(\Omega)}
\|\EE\|_{\CC^1(0,T;\LL_\infty(\Omega))}^2\|\PPsi_h\|_{\eps_0}^2.
\label{eq:estlhs}
\end{align}
For the right-hand side, we have:
\begin{align}
&-\eps_0\int_\Omega\left[\delta_{12}+\delta_{21}+\delta_{23}
+\delta_{31}+\delta_{32}+\delta_{34}\right]^\top\PPsi_hd\xx\nonumber\\
&=-\eps_0\int_\Omega\left[(1+\chi^{(1)})\partial_t(\PP_h\EE-\EE)^\top\PPsi_h\right.\nonumber\\
&\quad +\chi^{(3)}(\EE_h+\EE)^\top(\EE_h-\EE)\partial_t\EE_h^\top\PPsi_h
+\chi^{(3)}|\EE|^2\partial_t(\PP_h\EE-\EE)^\top\PPsi_h\nonumber\\
&\quad +2\chi^{(3)}\left((\EE_h-\EE)\EE_h^\top\partial_t\EE_h\right)^\top\PPsi_h
+2\chi^{(3)}\left(\EE(\EE_h-\EE)^\top\partial_t\EE_h\right)^\top\PPsi_h\nonumber\\
&\quad +\left.
2\chi^{(3)}\left(\EE\EE^\top\partial_t(\PP_h\EE-\EE)\right)^\top\PPsi_h
\right]d\xx\displaybreak[0]\nonumber\\
&\le\eps_0\int_\Omega\left[(1+\chi^{(1)})|\partial_t(\PP_h\EE-\EE)||\PPsi_h|\right.\nonumber\\
&\quad +\chi^{(3)}|\EE_h+\EE||\EE_h-\EE||\partial_t\EE_h||\PPsi_h|
+\chi^{(3)}|\EE|^2|\partial_t(\PP_h\EE-\EE)||\PPsi_h|\nonumber\\
&\quad +2\chi^{(3)}\EE_h^\top\partial_t\EE_h(\EE_h-\EE)^\top\PPsi_h
+2\chi^{(3)}(\EE_h-\EE)^\top\partial_t\EE_h\EE^\top\PPsi_h\nonumber\\
&\quad +\left.
2\chi^{(3)}\EE^\top\partial_t(\PP_h\EE-\EE)\EE^\top\PPsi_h
\right]d\xx\displaybreak[0]\nonumber\\
&\le\eps_0\int_\Omega\left[(1+\chi^{(1)}+\chi^{(3)}|\EE|^2)|\partial_t(\PP_h\EE-\EE)|
|\PPsi_h|\right.\nonumber\\
&\quad +\chi^{(3)}|\EE_h+\EE||\partial_t\EE_h||\PPsi_h|^2
+\chi^{(3)}|\EE_h+\EE||\PP_h\EE-\EE||\partial_t\EE_h||\PPsi_h|\nonumber\\
&\quad +2\chi^{(3)}\EE_h^\top\partial_t\EE_h|\PPsi_h|^2
+2\chi^{(3)}\EE_h^\top\partial_t\EE_h(\PP_h\EE-\EE)^\top\PPsi_h\nonumber\\
&\quad +2\chi^{(3)}\PPsi_h^\top\partial_t\EE_h\EE^\top\PPsi_h
+2\chi^{(3)}(\PP_h\EE-\EE)^\top\partial_t\EE_h\EE^\top\PPsi_h\nonumber\\
&\quad +\left.
2\chi^{(3)}\EE^\top\partial_t(\PP_h\EE-\EE)\EE^\top\PPsi_h
\right]d\xx\displaybreak[0]\nonumber\\
&\le\eps_0\int_\Omega\left[(1+\chi^{(1)}+\chi^{(3)}|\EE|^2)|\partial_t(\PP_h\EE-\EE)|
|\PPsi_h|\right.\nonumber\\
&\quad +\chi^{(3)}|\EE_h||\partial_t\EE_h||\PPsi_h|^2
+\chi^{(3)}|\EE||\partial_t\EE_h||\PPsi_h|^2\nonumber\\
&\quad +\chi^{(3)}|\EE_h||\PP_h\EE-\EE||\partial_t\EE_h||\PPsi_h|
+\chi^{(3)}|\EE||\PP_h\EE-\EE||\partial_t\EE_h||\PPsi_h|\nonumber\\
&\quad +2\chi^{(3)}|\EE_h||\partial_t\EE_h||\PPsi_h|^2
+2\chi^{(3)}|\EE_h||\partial_t\EE_h||\PP_h\EE-\EE||\PPsi_h|\nonumber\\
&\quad +2\chi^{(3)}|\EE||\partial_t\EE_h||\PPsi_h|^2
+2\chi^{(3)}|\EE||\partial_t\EE_h||\PP_h\EE-\EE||\PPsi_h|\nonumber\\
&\quad +\left.
2\chi^{(3)}|\EE|^2|\partial_t(\PP_h\EE-\EE)||\PPsi_h|
\right]d\xx\displaybreak[0]\nonumber\\
&=\eps_0\int_\Omega\left[
(1+\chi^{(1)}+3\chi^{(3)}|\EE|^2)|\partial_t(\PP_h\EE-\EE)|
|\PPsi_h|\right.\nonumber\\
&\quad +3\chi^{(3)}|\EE_h||\partial_t\EE_h||\PP_h\EE-\EE||\PPsi_h|
+3\chi^{(3)}|\EE||\partial_t\EE_h||\PP_h\EE-\EE||\PPsi_h|\nonumber\\
&\quad +\left.3\chi^{(3)}|\EE_h||\partial_t\EE_h||\PPsi_h|^2
+3\chi^{(3)}|\EE||\partial_t\EE_h||\PPsi_h|^2
\right]d\xx\displaybreak[0]\nonumber\\
&=\eps_0\int_\Omega\left[
(1+\chi^{(1)}+3\chi^{(3)}|\EE|^2)|\partial_t(\PP_h\EE-\EE)|
|\PPsi_h|\right.\nonumber\\
&\quad +3\chi^{(3)}(|\EE_h|+|\EE|)|\partial_t\EE_h||\PP_h\EE-\EE||\PPsi_h|
+\left.3\chi^{(3)}(|\EE_h|+|\EE|)|\partial_t\EE_h||\PPsi_h|^2
\right]d\xx\nonumber\\
&\le \left[\|1+\chi^{(1)}\|_{L_\infty(\Omega)}
+3\|\chi^{(3)}\|_{L_\infty(\Omega)}\|\EE\|_{\CC(0,T;\LL_\infty(\Omega))}^2
\right]\|\partial_t(\PP_h\EE-\EE)\|_{\eps_0}\|\PPsi_h\|_{\eps_0}\nonumber\\
&\quad +3\|\chi^{(3)}\|_{L_\infty(\Omega)}
\left[\|\EE_h\|_{\CC(0,T;\LL_\infty(\Omega))}+\|\EE\|_{\CC(0,T;\LL_\infty(\Omega))}\right]
\|\partial_t\EE_h\|_{\CC(0,T;\LL_\infty(\Omega))}
\|\PP_h\EE-\EE\|_{\eps_0}\|\PPsi_h\|_{\eps_0}\nonumber\\
&\quad +3\|\chi^{(3)}\|_{L_\infty(\Omega)}
\left[\|\EE_h\|_{\CC(0,T;\LL_\infty(\Omega))}+\|\EE\|_{\CC(0,T;\LL_\infty(\Omega))}\right]
\|\partial_t\EE_h\|_{\CC(0,T;\LL_\infty(\Omega))}
\|\PPsi_h\|_{\eps_0}^2\nonumber\\
&=: C_1\|\partial_t(\PP_h\EE-\EE)\|_{\eps_0}\|\PPsi_h\|_{\eps_0}
+C_2\|\PP_h\EE-\EE\|_{\eps_0}\|\PPsi_h\|_{\eps_0}
+C_3\|\PPsi_h\|_{\eps_0}^2,
\label{eq:estrhs}
\end{align}
where the positive constants $C_1$, $C_2$, $C_3$ depend on certain norms
of $\chi^{(1)}$, $\chi^{(3)}$, $\EE$, and $\EE_h$.
Combining the estimates (\ref{eq:estlhs}) and (\ref{eq:estrhs}) with (\ref{eq:3.6eb}), we get
\begin{align*}
&\frac{1}{2}\partial_t\|\PPsi_h\|_{\eps_0(1+\chi^{(1)})}^2
+\frac{\eps_0}{2}\partial_t\int_\Omega\chi^{(3)}\left[|\EE|^2|\PPsi_h|^2
+2|\EE^\top\PPsi_h|^2\right]d\xx\\
&\quad -3\|\chi^{(3)}\|_{L_\infty(\Omega)}
\|\EE\|_{\CC^1(0,T;\LL_\infty(\Omega))}^2\|\PPsi_h\|_{\eps_0}^2
-(\nabla\times(\HH_h-\PPi_h\HH),\PPsi_h)\\
&\le\eps_0\int_\Omega\left[\delta_{11}+\delta_{22}+\delta_{33}\right]^\top\PPsi_hd\xx
-(\nabla\times(\HH_h-\PPi_h\HH),\PPsi_h)\\
&=-\eps_0\int_\Omega\left[\delta_{12}+\delta_{21}+\delta_{23}
+\delta_{31}+\delta_{32}+\delta_{34}\right]^\top\PPsi_hd\xx\\
&\le C_1\|\partial_t(\PP_h\EE-\EE)\|_{\eps_0}\|\PPsi_h\|_{\eps_0}
+C_2\|\PP_h\EE-\EE\|_{\eps_0}\|\PPsi_h\|_{\eps_0}
+C_3\|\PPsi_h\|_{\eps_0}^2.
\end{align*}
This finally leads to
\begin{align*}
&\frac{1}{2}\partial_t\|\PPsi_h\|_{\eps_0(1+\chi^{(1)})}^2
+\frac{\eps_0}{2}\partial_t\int_\Omega\chi^{(3)}\left[|\EE|^2|\PPsi_h|^2
+2|\EE^\top\PPsi_h|^2\right]d\xx\\
&\quad -(\nabla\times(\HH_h-\PPi_h\HH),\PPsi_h)\\
&\le \left[C_1\|\partial_t(\PP_h\EE-\EE)\|_{\eps_0}
+C_2\|\PP_h\EE-\EE\|_{\eps_0}\right]\|\PPsi_h\|_{\eps_0}
+C_4\|\PPsi_h\|_{\eps_0}^2,
\end{align*}
where
\[
C_4:=C_3+3\|\chi^{(3)}\|_{L_\infty(\Omega)}\|\EE\|_{\CC^1(0,T;\LL_\infty(\Omega))}^2.
\]
Now we consider (\ref{3.6eb}) with
$\PPhi_h=\HH_h-\PPi_h\HH$ and get
\begin{align*}
\frac{1}{2}\partial_t\|\PPhi_h\|_{\mu_0}^2
+(\EE_h-\PP_h\EE,\nabla\times\PPhi_h)
&=\mu_0(\partial_t\HH-\PPi_h\partial_t\HH,\PPhi_h)\\
&\le \|\partial_t\HH-\PPi_h\partial_t\HH\|_{\mu_0}\|\PPhi_h\|_{\mu_0}.
\end{align*}
Adding both inequalities and making use of the commutation property
of $\PP_h$, we arrive at
\begin{align*}
&\frac{1}{2}\partial_t\|\PPsi_h\|_{\eps_0(1+\chi^{(1)})}^2
+\frac{1}{2}\partial_t\|\PPhi_h\|_{\mu_0}^2
+\frac{\eps_0}{2}\partial_t\int_\Omega\chi^{(3)}\left[|\EE|^2|\PPsi_h|^2
+2|\EE^\top\PPsi_h|^2\right]d\xx\\
&\le \left[C_1\|\partial_t\EE-\PP_h\partial_t\EE\|_{\eps_0}
+C_2\|\EE-\PP_h\EE\|_{\eps_0}\right]\|\PPsi_h\|_{\eps_0}
+\|\partial_t\HH-\PPi_h\partial_t\HH\|_{\mu_0}\|\PPhi_h\|_{\mu_0}\\
&\quad +C_4\|\PPsi_h\|_{\eps_0}^2.
\end{align*}
The projection errors can be estimated by means of (\ref{eq:errL2projection})
and (\ref{eq:errorpi1}), that is, for
$\EE,\partial_t\EE\in\HH^k(\Omega)$ and $\partial_t\HH\in\HH^{k+1}(\Omega)$,
we have that
\begin{align*}
\|\EE-\PP_h\EE\|_{\eps_0}
&\le C\sqrt{\eps_0}\,h^k\|\EE\|_{\HH^k(\Omega)}
\le C\sqrt{\eps_0}\,h^k\|\EE\|_{\CC(0,T;\HH^k(\Omega))},\\
\|\partial_t\EE-\PP_h\partial_t\EE\|_{\eps_0}
&\le C\sqrt{\eps_0}\,h^k\|\partial_t\EE\|_{\HH^k(\Omega)}
\le C\sqrt{\eps_0}\,h^k\|\partial_t\EE\|_{\CC(0,T;\HH^k(\Omega))},\\
\|\partial_t\HH-\PPi_h\partial_t\HH\|_{\mu_0}
&\le C\sqrt{\mu_0}\,h^k\|\partial_t\HH\|_{\HH^{k+1}(\Omega)}
\le C\sqrt{\mu_0}\,h^k\|\partial_t\HH\|_{\CC(0,T;\HH^{k+1}(\Omega))}.
\end{align*}
In this way the above estimate can be written as
\begin{align*}
&\frac{1}{2}\partial_t\|\PPsi_h\|_{\eps_0(1+\chi^{(1)})}^2
+\frac{1}{2}\partial_t\|\PPhi_h\|_{\mu_0}^2
+\frac{\eps_0}{2}\partial_t\int_\Omega\chi^{(3)}\left[|\EE|^2|\PPsi_h|^2
+2|\EE^\top\PPsi_h|^2\right]d\xx\\
&\le C_5h^k\left[\|\PPsi_h\|_{\eps_0}+\|\PPhi_h\|_{\mu_0}\right]
+C_4\|\PPsi_h\|_{\eps_0}^2.
\end{align*}
Setting
\[
w_h(t):=\sqrt{\|\PPsi_h(t)\|_{\eps_0}^2+\|\PPhi_h(t)\|_{\mu_0}^2},
\]
we get
\begin{align*}
&\frac{1}{2}\partial_t\|\PPsi_h\|_{\eps_0(1+\chi^{(1)})}^2
+\frac{1}{2}\partial_t\|\PPhi_h\|_{\mu_0}^2
+\frac{\eps_0}{2}\partial_t\int_\Omega\chi^{(3)}\left[|\EE|^2|\PPsi_h|^2
+2|\EE^\top\PPsi_h|^2\right]d\xx\\
&\le C_5\sqrt{2}h^kw_h(t)
+C_4\|\PPsi_h\|_{\eps_0}^2\\
&\le C_5\sqrt{2}h^kw_h(t)
+C_4w_h^2(t).
\end{align*}
Integrating this inequality, we obtain
\begin{align}
&\frac{1}{2}\|\PPsi_h(t)\|_{\eps_0(1+\chi^{(1)})}^2
+\frac{1}{2}\|\PPhi_h(t)\|_{\mu_0}^2 \nonumber\\
&\quad +\frac{\eps_0}{2}\int_\Omega\chi^{(3)}\left[|\EE(t)|^2|\PPsi_h(t)|^2
+2|\EE(t)^\top\PPsi_h(t)|^2\right]d\xx \nonumber\\
&\le\frac{1}{2}\|\PPsi_h(0)\|_{\eps_0(1+\chi^{(1)})}^2
+\frac{1}{2}\|\PPhi_h(0)\|_{\mu_0}^2 \nonumber\\
&\quad +\frac{\eps_0}{2}\int_\Omega\chi^{(3)}\left[|\EE(0)|^2|\PPsi_h(0)|^2
+2|\EE(0)^\top\PPsi_h(0)|^2\right]d\xx \nonumber\\
&\quad +\int_0^t\left[C_5\sqrt{2}h^k w_h(s)+C_4w_h^2(s)\right]ds.
\label{eq:gd}
\end{align}
By the monotonicity of the weighted norms w.r.t.\ the weight and
the nonnegativity of the integral term on the left-hand side, we see that
\begin{align}
\frac12 w_h^2(t)&\le \frac{1}{2}\|\PPsi_h(t)\|_{\eps_0(1+\chi^{(1)})}^2
+\frac{1}{2}\|\PPhi_h(t)\|_{\mu_0}^2 \nonumber\\
&\quad +\frac{\eps_0}{2}\int_\Omega\chi^{(3)}\left[|\EE(t)|^2|\PPsi_h(t)|^2
+2|\EE(t)^\top\PPsi_h(t)|^2\right]d\xx.
\label{eq:lhsgd}
\end{align}
On the other hand, we have the estimates
\begin{align}
\label{eq:rhsgd1}
\|\PPsi_h(0)\|_{\eps_0(1+\chi^{(1)})}^2
\le \|1+\chi^{(1)}\|_{L_\infty(\Omega)}\|\PPsi_h(0)\|_{\eps_0}^2
\le \|1+\chi^{(1)}\|_{L_\infty(\Omega)} w_h^2(0)
\end{align}
and
\begin{align}
&\eps_0\int_\Omega\chi^{(3)}\left[|\EE(0)|^2|\PPsi_h(0)|^2
+2|\EE(0)^\top\PPsi_h(0)|^2\right]d\xx \nonumber\\
&\le 3\|\chi^{(3)}\|_{L_\infty(\Omega)}\|\EE(0)\|_{L_\infty(\Omega)}^2
\|\PPsi_h(0)\|_{\eps_0}^2 \nonumber\\
&\le 3\|\chi^{(3)}\|_{L_\infty(\Omega)}\|\EE(0)\|_{L_\infty(\Omega)}^2
w_h^2(0).
\label{eq:rhsgd2}
\end{align}
Combining (\ref{eq:lhsgd}), (\ref{eq:rhsgd1}), (\ref{eq:rhsgd2})
with (\ref{eq:gd}), we get
\begin{align*}
\frac12 w_h^2(t)
&\le \frac12 \|1+\chi^{(1)}\|_{L_\infty(\Omega)} w_h^2(0)
+\frac32\|\chi^{(3)}\|_{L_\infty(\Omega)}\|\EE(0)\|_{L_\infty(\Omega)}^2
w_h^2(0)\\
&\quad +\int_0^t\left[C_5\sqrt{2}h^k w_h(s)+C_4w_h^2(s)\right]ds,
\end{align*}
or, equivalently,
\begin{align}\label{eq:gd2}
w_h^2(t) \le C_6^2 w_h^2(0) + \int_0^t\left[2C_5\sqrt{2}h^k w_h(s)+2C_4w_h^2(s)\right]ds,
\end{align}
where $C_6^2:=\|1+\chi^{(1)}\|_{L_\infty(\Omega)}
+3\|\chi^{(3)}\|_{L_\infty(\Omega)}\|\EE(0)\|_{L_\infty(\Omega)}^2$.

In the paper \cite{Dafermos:79}, a Gronwall-type lemma (Lemma 4.1) is specified which gives
a bound on the value $w(T)$ provided an inequality like (\ref{eq:gd2}) is satisfied:
\begin{align*}
w_h(T)\le C_6 e^{C_4 T}w_h(0)+ C_5\sqrt{2}h^kTe^{C_4 T}.
\end{align*}
From this and the triangle inequality in conjunction with
(\ref{eq:errL2projection}) and (\ref{eq:errorpi1})
the statement follows.
Indeed, since
\[
w(t)\le\|\PPsi_h(t)\|_{\eps_0}+\|\PPhi_h(t)\|_{\mu_0}\le\sqrt{2}w(t)
\quad\text{for all }t\in[0,T],
\]
we get
\begin{align*}
\|\PPsi_h(T)\|_{\eps_0}+\|\PPhi_h(T)\|_{\mu_0}
&\le\sqrt{2}\left[
C_6 w_h(0)+ C_5\sqrt{2}h^kT
\right]e^{C_4 T}\\
&\le\sqrt{2}\left[
C_6 \|\PPsi_h(0)\|_{\eps_0}+ C_6\|\PPhi_h(0)\|_{\mu_0}+ C_5\sqrt{2}h^kT
\right]e^{C_4 T}.
\end{align*}
Then
\begin{align*}
&\|\EE_h(T)-\EE(T)\|_{\eps_{0}}+\|\HH_h(T)-\HH(T)\|_{\mu_{0}}\\
&\le \|\PPsi_h(T)\|_{\eps_0}+\|\PPhi_h(T)\|_{\mu_0}
+\|\PP_h\EE(T)-\EE(T)\|_{\eps_{0}}+\|\PPi_h\HH(T)-\HH(T)\|_{\mu_{0}}\\
&\le\sqrt{2}\left[
C_6 \|\PPsi_h(0)\|_{\eps_0}+ C_6\|\PPhi_h(0)\|_{\mu_0}+ C_5\sqrt{2}h^kT
\right]e^{C_4 T}\\
&\quad + Ch^k\left[\sqrt{\eps_0}\|\EE\|_{\CC(0,T;\HH^k(\Omega))}
+ \sqrt{\mu_0}\|\HH\|_{\CC(0,T;\HH^{k+1}(\Omega))}\right],
\end{align*}
which implies the stated estimate.
\end{proof}
\begin{remark}
Note that the constant $C$ in this estimate behaves as $Te^{C_4 T}$
for large $T$.
\end{remark}

\section{Conclusion}
In this paper we have investigated a semi-discrete conforming finite element approximation
to the solution of Maxwell's equations for nonlinear media of Kerr-type
using N\'ed\'elec elements from the first family.
We have demonstrated energy (stability) estimates for the weakly formulated
problem and error estimates for the semi-discretized problem.
The results can be extended to other conforming finite element methods
provided the corresponding projection operators $\PP_h$ and $\PPi_h$
((\ref{eq:L2projection}), (\ref{eq:pi1a})--(\ref{eq:pi1b}))
admit analogous properties.

\bibliographystyle{alpha}

\end{document}